\documentclass[11pt, a4paper]{amsart}

\usepackage{amsmath,amssymb,amsfonts, float}
\usepackage[all]{xy}
\usepackage{caption}
\usepackage{enumerate}
\usepackage{mathpazo}
\usepackage{a4wide}
\usepackage{diagbox}

\usepackage{bm}
\usepackage{graphicx}
\usepackage[breaklinks=true]{hyperref}

\usepackage{xcolor}
\usepackage{multicol}
\usepackage{tabularx}

\usepackage{hyperref}
\usepackage[capitalize]{cleveref}

\usepackage[normalem]{ulem}

\newtheorem{thm}{Theorem}[section] 
\newtheorem*{thm*}{Theorem} 
\newtheorem{prop}[thm]{Proposition}
\newtheorem{lem}[thm]{Lemma}
\newtheorem{cor}[thm]{Corollary}

\theoremstyle{definition}
\newtheorem{definition}[thm]{Definition}
\newtheorem{expl}[thm]{Example}

\newtheorem{rem}[thm]{Remark}

\DeclareMathOperator{\C}{\mathbb{C}}
\DeclareMathOperator{\Z}{\mathbb{Z}}

\DeclareMathOperator{\F}{\mathbb{F}}

\DeclareMathOperator{\Q}{\mathbb{Q}}

\DeclareMathOperator{\Gal}{\text{Gal}}
\DeclareMathOperator{\Hom}{{\rm Hom}}

\DeclareMathOperator{\OO}{\mathcal{O}}

    \DeclareFontFamily{U}{wncy}{}
    \DeclareFontShape{U}{wncy}{m}{n}{<->wncyr10}{}
    \DeclareSymbolFont{mcy}{U}{wncy}{m}{n}
    \DeclareMathSymbol{\Sha}{\mathord}{mcy}{"58}

\numberwithin{equation}{section}

\DeclareSymbolFont{bbold}{U}{bbold}{m}{n}
\DeclareSymbolFontAlphabet{\mathbbold}{bbold}

\usepackage{hyperref}

\begin{document}
\title{Integral Cayley graphs \\ over a finite symmetric algebra}

 \author{Tung T. Nguyen, Nguy$\tilde{\text{\^{e}}}$n Duy T\^{a}n }

 \address{Department of Mathematics and Computer Science, Lake Forest College, Lake Forest, Illinois, USA}
 \email{tnguyen@lakeforest.edu}
 
  \address{
Faculty Mathematics and Informatics, Hanoi University of Science and Technology, 1 Dai Co Viet Road, Hanoi, Vietnam } 
\email{tan.nguyenduy@hust.edu.vn}

\thanks{TTN is partially supported by an AMS-Simons Travel Grant.  NDT is partially supported by the Vietnam National
Foundation for Science and Technology Development (NAFOSTED) under grant number 101.04-2023.21}

\begin{abstract}
A graph is called integral if its eigenvalues are integers. In this article, we provide the necessary and sufficient conditions for a Cayley graph over a finite symmetric algebra $R$ to be integral. This generalizes the work of So who studies the case where $R$ is the ring of integers modulo $n.$ We also explain some number-theoretic constructions of finite symmetric algebras arising from global fields, which we hope could pave the way for future studies on Paley graphs associated with a finite Hecke character.
\end{abstract}
\maketitle
\section{Introduction}

An undirected graph $G$ is said to be integral if all of its eigenvalues are integers. The notion of integral graphs was first introduced by Harary and Schwenk in \cite{harary1974graphs}. In the same article, the authors asked whether one can classify integral graphs. Since then, there has been a vast literature on this topic. We refer interested readers to \cite{balinska2002survey} for a survey about known examples of integral graphs.

While the general question is quite challenging, the situation becomes more manageable when we consider graphs with additional structures, such as Cayley graphs over a finite ring. In this case, we can exploit the interplay between the additive and multiplicative structures of the ring to study the arithmetics of these Cayley graphs. As one might naturally expect, this investigation bridges various fields, including number theory, character theory, and commutative algebra. Below, we provide further details about some old and recent studies for this line of research.

In \cite{klotz2007some}, Klotz and Sander study unitary Cayley graphs over $\Z/n$ and their natural generalization: the gcd-graphs $G_{n}(D)$ where $D$ is a subset of proper divisors of $n$. Let us quickly recall the definition of $G_n(D)$. The vertices of $G_n(D)$ are elements of the finite ring $\Z/n$ and two elements $a,b$ are adjacent if $\gcd(a-b, n) \in D.$ Using the theory of Ramanujan sums, Klotz and Sander show that the gcd-graphs $G_n(D)$ are integral. They also ask whether the converse is also true; namely, if a Cayley graph over $\Z/n$ is integral, is it true that it is a gcd-graph? 
 So (2016) in \cite{so2006integral} provides an affirmative answer to this question. For the reader's convenience, we recall So's theorem (\cite[Theorem 7.1]{so2006integral}) in terms of Cayley graphs over $\Z/n$  as follows. 
\begin{thm*}[So's theorem]
    A undirected Cayley graph $G=\Gamma(\Z/n,S(G))$ is integral if and only if $S(G)$ is a union of the $G_n(d)'s$.
\end{thm*}
Here for each $d\mid n$,  \[ G_n(d) = \{m \in \Z/n\mid  \gcd(m,n)=d \}. \] 

In \cite{minavc2024gcd}, inspired by the analogy between number fields and function fields, we study gcd-graphs over polynomial rings. Among various things that we find, we show that the spectrum of these gcd-graphs also has an explicit formula via Ramanujan sums. An important insight of our work is the notion of a symmetric $\F_p$-algebra (see \cite[Definition 6.2]{minavc2024gcd}).  By their very definition, symmetric algebras are self-dual, and consequently, their characters are parametrized by their elements, which allows us to calculate the associated Gauss and Ramanujan sums quite explicitly.

In this article, we apply this circle of ideas to study integral graphs over a finite ring. Our main goal here is to generalize So's theorem to finite symmetric algebras (we refer the reader to the \cref{def:symmetric} for the precise definition of a symmetric algebra.) More precisely, we prove the following.

\begin{thm}
    Let $R$ be a finite symmetric $\Z/n$-algebra and $S \subset R \setminus \{0\}$ such that $S=-S$. Then the undirected graph $\Gamma(R,S)$ is integral if and only if $S$ is stable under the action of $(\Z/n)^{\times}.$
\end{thm}
We remark that our main theorem works for all $S$ but since we only work with undirected graphs, we impose the above conditions on $S$.  When $R=\Z/n$, this recovers So's theorem. We note, however, that our theorem applies to a much wider class of finite rings. More precisely, as we will show in the last section, all finite quotients of the ring of integers in a global field are symmetric. We believe that this observation could pave the way for future studies on gcd-graphs over such rings as well as Paley graphs associated with finite Hecke characters (see \cref{rem:future_work}). 

\section{Main results}
Let $R$ be a finite $\Z/n$-algebra. In this paper, we study Cayley graphs of the form $\Gamma(R,S)$ where $S$ is a symmetric subset of $R \setminus \{0\}$. We recall that $\Gamma(R,S)$ is the undirected graph equipped with the following data. 
\begin{enumerate}
    \item The vertex set of $\Gamma(R,S)$ is $R.$
    \item Two elements $a,b \in R$ are adjacent if there exists $s \in S$ such that $b=a+s.$
\end{enumerate}
We are interested in the case where $\Gamma(R,S)$ is integral; i.e, all of their eigenvalues are integers. By the circulant diagonalization theorem for finite abelian groups (see \cite{kanemitsu2013matrices}), the spectrum of $\Gamma(R,S)$ is given by the family 
\[ \left\{ \sum_{s \in S} \widehat{\psi}(s)\right \}_{\widehat{\psi}}, \]
where $\widehat{\psi}$ runs over the dual group $\widehat{R}:= {\rm Hom}(R, \C^{\times})$ of all characters of $(R, +)$. We remark that since $R$ is an $\Z/n$-algebra, $\widehat{\psi}(s)^n=1$ for all $\widehat{\psi} \in \widehat{R}$ and $s \in R.$ As a result, once we fix a primitive $n$th root of unity $\zeta_n \in \C^{\times}$, an element $\widehat{\psi} \in \widehat{R}$ can be expressed uniquely in the form 
\[ \widehat{\psi}(s) = \zeta_n^{\psi(s)} ,\]
where $\psi\colon R \to \Z/n$ is a group homomorphism. For the rest of our discussion, we will use this formulation to identify $\widehat{\psi}$ and $\psi$. 

Since $R$ is an $\Z/n$-algebra, $(\Z/n)^{\times}$ acts naturally on $R$. The following proposition gives a sufficient condition for $\Gamma(R,S)$ to be integral. 
\begin{prop} \label{prop:stable}
Suppose that $S$ is stable under the action of $(\Z/n)^{\times}.$ Then $\Gamma(R,S)$ is an integral graph. 
\end{prop}

\begin{proof}
The Galois group of $\Q(\zeta_n)/\Q$ is naturally isomorphic to $(\Z/n)^{\times}.$    More precisely, the map  
\[ \chi: (\Z/n)^\times\to \Gal(\Q(\zeta_n)/\Q), \]
defined by sending $a \mapsto \sigma_a$, where $\sigma_a(\zeta_n) = \zeta_n^a$, is an isomorphism. For each $\psi \in \Hom(R, \Z/n)$ and $a\in (\Z/n)^\times$, we have
\[ \sigma_a \left(\sum_{s \in S} \widehat{\psi}(s) \right) = \sigma_a \left(\sum_{s \in S} \zeta_n^{{\psi}(s)}. \right) = \sum_{s \in S} \zeta_n^{a \psi(s)} = \sum_{s \in S} \zeta_n^{\psi(as)} = \sum_{s \in S} \zeta_n^{\psi(s)}. \]
The last equality follows from the fact that $S$ is stable under the action of $(\Z/n)^{\times}$. This shows that $\sum_{s\in S} \zeta_n^{\psi(s)} \in \Z.$ We conclude that $\Gamma(R,S)$ is integral. 
\end{proof}

Our goal is to study the converse of \cref{prop:stable}. In this article, we provide a partial answer to this question. More precisely, we will show that the converse of \cref{prop:stable} holds for the class of symmetric $\Z/n$-algebras, whose definition we now recall. 

\begin{definition} \label{def:symmetric}
A finite $\Z/n$-algebra $R$ is called symmetric if there exists $\psi \in \Hom(R, \Z/n)$ such that the kernel of $\psi$ does not contain any non-zero ideal of $R.$ When $\psi$ exists, we call it a non-degenerate linear functional on $R.$
\end{definition}

\begin{expl}

$\Z/n$ is a symmetric $\Z/n$-algebra, where the identity map from $\Z/n$ to itself is a non-degenerate linear functional. More generally, as we will show in \cref{sec:example}, a finite quotient of the ring of integers in a global field is a symmetric algebra. 
\end{expl}

Let $R$ be a finite symmetric $\Z/n$-algebra equipped with a fixed linear functional $\psi\colon R \to \Z/n.$ For each $r \in R$, we can define   
\[ \psi_r \in \Hom(R, \Z/n), \]
 by the rule $\psi_r(t) = \psi(rt)$. Furthermore, the map 
\[ \Phi\colon R \to \Hom(R, \Z/n), \]
defined by sending $r \mapsto \psi_r$ is a group homomorphism. Since $\psi$ is non-degenerate, $\Phi$ is injective, hence surjective. In summary, we have the following proposition. 

\begin{prop}
Let $R$ be a finite symmetric $\Z/n$-algebra with a fixed non-degenerate linear functional $\psi.$ Then for each character $\widehat{\psi}$ of $R$, there exists a unique element $r\in R$ such that for all $t \in R$
\[ \widehat{\psi}(t) = \zeta_n^{\psi_r(t)} = \zeta_n^{\psi(rt)}.\]
\end{prop}
For the rest of the article, we will implicitly fix an indexing of elements in $R$. Let $O_1, O_2, \ldots, O_d$ be the orbits of $R$ under the action of $(\Z/n)^{\times}.$ Let $\bm{v}_i \in \Q^{|R|}$ be the characteristics vector of $O_i$; namely 
\[ \bm{v}_i[r] = 
\begin{cases}
    1, & \text{if } r \in O_i \\
    0, & \text{if } r \not \in O_i.
\end{cases}
\]
By definition, $\bm{v}_1, \bm{v}_2, \ldots, \bm{v}_d$ are linearly independent over $\Q$. Let $\mathcal{V}$ be the $\Q$-vector space generated by the $\bm{v}_i$'s.

Let $A_{R}=(\zeta_n^{\psi_r(t)})_{r, t \in R} = (\zeta_n^{\psi(rt)})_{ r,t \in R}$ be the DFT matrix associated with $R$ (see \cite{kanemitsu2013matrices}). Then for each $R$-circulant matrix $C$ formed by a $1 \times |R|$-vector $\bm{v}$, the spectrum of $C$ is given by the vector $A_{R} \bm{v}.$ In particular, the spectrum of $\Gamma(R,S)$ is given by $A_{R} \bm{1}_{S}$ where $\bm{1}_{S}$ is the characteristic vector of $S.$

Following \cite{so2006integral}, we define the following vector space. 

\[ \mathcal{A} = \{\bm{v} \in \Q^{|R|} \mid A_{R} \bm{v} \in \Q^{|R|}\} .\]

By the proof of \cref{prop:stable}, we know that $\mathcal{V} \subset \mathcal{A}.$
\begin{prop} \label{prop:projection}
If $\bm{v} \in \mathcal{A}$ then $A_{R} \bm{v} \in \mathcal{V}.$
\end{prop}

\begin{proof}
Let $\bm{v}=(v_t)_{t \in R} \in \Q^{|R|}$ and $A_R \bm{v} = (u_r)_{r \in R}.$ Because $\bm{v} \in \mathcal{A}$, $u_r \in \Q$ for all $r$. By definition 
\[ u_r = \sum_{t \in R} \zeta_n^{\psi(rt)} v_t.\]
We claim that if $r_1,r_2$ belong to the same equivalence class ($r_1,r_2$ belong to the same orbit) then $u_{r_1} = u_{r_2}.$ This would imply that $A_R \bm{v} \in \mathcal{V}$. To prove this statement, we remark that since $r_1, r_2$ belong to the same equivalence class, we can find $a \in (\Z/n)^{\times}$ such that $ar_1 = r_2.$ We then have 
\[ u_{r_1} = \sigma_a(u_{r_1}) =  \sum_{t \in R} \zeta_n^{\psi(ar_1t)}v_t =  \sum_{t \in R} \zeta_n^{\psi(r_2t)} v_t = u_{r_2}.
\qedhere\]
\end{proof}

\begin{cor} \label{cor:equal}
    $\mathcal{A} = \mathcal{V}.$ In particular $\{\bm{v}_i\}_{i=1}^d$ forms a basis for $\mathcal{A}.$
\end{cor}

\begin{proof}
    By \cref{prop:projection} we know that $A_{R} \mathcal{A} \subset \mathcal{V}.$ Since $A_R$ is invertible, we conclude that $\dim(\mathcal{A}) \leq \dim(\mathcal{V}).$ Since $\mathcal{V} \subset \mathcal{A}$, we conclude that $\mathcal{A} =\mathcal{V}.$
\end{proof}

\begin{thm} \label{thm:main}
Let $R$ be a finite symmetric $\Z/n$-algebra and  $S \subset R$. Then $\Gamma(R,S)$ is an integral graph if and only if $S$ is stable under the action of $(\Z/n)^{\times}.$
\end{thm}

\begin{proof}
    The forward direction has been proved in \cref{prop:stable}. Let us prove the other direction. Suppose that $\Gamma(R,S)$ is integral. Then by \cref{cor:equal}, $\bm{1}_{S} \in \mathcal{V}$. Since the components of $\bm{1}_{S}$ are in the set $\{0,1\}$, we conclude that $S$ is a union of some of the orbits $O_1, O_2, \ldots, O_d$. In other words, $S$ is stable under the action of $(\Z/n)^{\times}.$
\end{proof}

\begin{rem}
    If $R = \Z/n$ then $R$ is a symmetric $\Z/n$-algebra where $\psi\colon R \to \Z/n$ is the identity map. Furthermore, we can see that the orbits of $(\Z/n)$ under the action of $(\Z/n)^{\times}$ are precisely $\{G_n(d)\}_{d \mid n}$ where 
    \[ G_n(d) = \{m \in \Z/n\mid  \gcd(m,n)=d \}. \]
    We then see that \cref{thm:main} is a generalization of \cite[Theorem 7.1]{so2006integral}.
\end{rem}

\section{Examples of symmetric algebras} \label{sec:example}
In this section, we provide some constructions of finite symmetric $\Z/n$-algebras. The first example is quite standard (see \cite[Example 3.15E]{lam2012lectures}).
\begin{expl} \label{expl:group_ring}
Let $G$ be a finite abelian group and $R=\Z/n[G]$ the group algebra of $G$ with coefficients in $\Z/n.$ Let $\psi\colon R \to \Z/n$ be the linear functional defined by 
\[ \psi\left(\sum_{g} a_g g \right) = a_{e}.\]
Then $\psi$ is non-degenerate and hence $R$ is a symmetric $\Z/n$-algebra. 
\end{expl}
We now show that all finite quotient rings of the ring of integers in a global field is a symmetric $\Z/n$-algebra where $n$ is the characteristic of the ring. We start this investigation with a series of simple lemmas. 

\begin{lem} \label{prop:product}
    If $R_1, R_2$ are two symmetric $\Z/n$-algebras then so is $R_1 \times R_2.$
\end{lem}

\begin{proof}
Let $\psi_1\colon R_1 \to \Z/n$ and $\psi_2\colon R_2 \to \Z/n$ be two non-degenrate linear functionals. Let $\psi\colon R_1 \times R_2 \to \Z/n$ be defined as 
\[ \psi(r_1,r_2) = \psi_1(r_1) + \psi_2(r_2).\]
We can check that $\psi$ is a non-degenerate $\Z/n$-linear functional on $R_1 \times R_2.$ By definition, $R_1 \times R_2$ is a symmetric $\Z/n$-algebra. 
\end{proof}

\begin{lem} \label{lem:product_coefficients}
If $R_1$ is a symmetric $\Z/n_1$-algebra and $R_2$ is a symmetric $\Z/n_2$-algebra with $\gcd(n_1, n_2)=1$, then $R_1 \times R_2$ is a symmetric algebra over $\Z/n$, where $n=n_1n_2$.
\end{lem}
\begin{proof}
    This follows from the Chinese remainder theorem and the proof of \cref{prop:product}.
\end{proof}

\begin{lem} \label{lem:induced}
    If $R$ is a symmetric $\Z/n$-algebra then $R$ is also a symmetric $\Z/m$-algebra for all $ n \mid m.$
\end{lem}

\begin{proof}
    There is an embedding of $ \iota_{n,m}\colon \Z/n \to \Z/m$ defined by $a \mapsto \frac{m}{n}a.$ Let $\psi\colon R \to \Z/n$ be a non-degenerate linear functional. Then the composition of $\psi$ with $\iota_{n,m}$ is a non-degenerate linear functional of $R$ over $\Z/m.$
\end{proof}

\begin{prop}  \label{prop:quotient_induced}
    Let $A$ be an integral domain. Let $f \in A$ such that $R=A/f$ is a finite ring of characteristics n. Suppose that $R$ is a symmetric $\Z/n$-algebra. Then for each $g \mid f$, $A/g$ is also a symmetric $\Z/n$-algebra. 
\end{prop}

\begin{proof}
Let $\psi_{f/g}\colon A/g \to \Z/n$ be  defined by 
    \[ \psi_{f/g}(a) = \psi\left(\frac{f}{g}a \right) ,\]
    We can see that $\psi_{f/g}$ is a linear function on $A/g.$ We claim that it is non-degenerate as well. Suppose to the contrary that the kernel of $\psi_{f/g}$ contains a non-zero ideal $I \subset A/g$. Let $\bar{a}$ be an arbitrary element in $I$ and let $a$ be a lift of $\bar{a}$ to $A/f.$ By definition, $\langle \bar{a} \rangle \subset I.$ As a result, for all $b \in A/f$
    \[ \psi \left(\frac{f}{g} \bar{a}b \right) = \psi\left(\frac{f}{g} ab \right) = 0 .\]
    This shows that the ideal generated by $\frac{f}{g}a$ belongs to the kernel of $\psi.$ Since $\psi$ is non-degenerate, we conclude that $\langle \frac{f}{g}a \rangle $ is the zero ideal in $A/f.$ As a result, we can find $h \in A$ such that  $\frac{f}{g} a= fh$. Since $A$ is an integral domain, this implies that $a = hg$ and hence $\bar{a} =0.$
\end{proof}
\begin{lem} \label{lem:local}
Let $K$ be a finite extension of $\Q_p$ and $\OO_K$ its ring of algebraic integers. Let $\mathcal{I} \subset \OO_K$ be a non zero ideal in $\OO_K$ and $p^{a}$ be the power of $p$ such that 
\[  \mathcal{I} \cap \Z_p = p^{a} \Z_p.\]
Then $\OO_K/\mathcal{I}$ is a symmetric $\Z/p^a$-algebra. 
\end{lem}

\begin{proof}
Since $\OO_K$ is a DVR, $\mathcal{I}$ is a principal ideal; namely $\mathcal{I} = \langle f \rangle $ for some $f \in \OO_K.$ We then have, $\OO_K/\mathcal{I} \cong \OO_K/f.$   By \cref{prop:quotient_induced}, it is enough to show that $\OO_K/p^a$ is a symmetric $\Z/p^a$-algebra. It is well-known that $\OO_K$ is monogenic over $\Z_p$, i.e, there is $\alpha \in \OO_K$ such that $\OO_K = \Z_p[\alpha]$ (see \cite[Chapter 3, Section 6, Proposition 12]{serre2013local}). Let $m = [K:\Q_p]$, then every element in $\OO_K/p^a$ can be written uniquely as $\sum_{i=0}^{m-1} a_i \alpha^i$ where $\alpha_i \in \Z_p/p^a = \Z/p^a.$ We can define a linear functional $ \psi\colon \OO_K/p^a \to \Z/p^a$ by 
    \[ \psi\left(\sum_{i=0}^{m-1} a_i \alpha^{i} \right)= a_{m-1}.\]
    By an identical argument for the proof of \cite[Proposition 6.7]{minavc2024gcd}, we can check that $\psi$ is non-degenerate and hence $\OO_K/p^a$ is a symmetric  $\Z/p^a$-algebra. 
\end{proof}

\begin{prop}
Let $K$ be a local field of characteristics $p$ and $\OO_K$ its ring of integers. For each  non-zero ideal $ \mathcal{I} \subset \OO_K$, $\OO_K/\mathcal{I}$ is a symmetric $\F_p$-algebra. 
\end{prop}

\begin{proof}
    A local field of characteristic $p$ is isomorphic to the Laurent series $\F_q((t))$ for some $q$. As a result, $\OO_K \cong \F_q[[t]]$. We can then find $a$ such that $\mathcal{I} = \langle t^a \rangle.$ We then have $\F_q[[t]]/\mathcal{I} \cong \F_{q}[[t]]/t^a \cong \F_q[t]/t^a.$ By \cite[Corollary 6.8]{minavc2024gcd}, $\F_q[t]/t^a$ is a symmetric $\F_p$-algebra. 
\end{proof}
\begin{thm} \label{thm:number_field}
Let $K$ be a number field, $\mathcal{I} \subset \OO_K$ be a non-zero ideal. Let $n$ be the positive integer such that 
\[ n \Z = \OO_K \cap \mathcal{I}.\]
Then $\OO_K/\mathcal{I}$ is a symmetric $\Z/n$-algebra. 
\end{thm}

\begin{proof}
    Let $\mathcal{I} = \prod_{i=1}^d P_i^{e_i}$ be the factorization of $\mathcal{I}$ into the product of distinct prime ideals in $\OO_K.$ By the Chinese remainder theorem 
    \[ \OO_K/ \mathcal{I} \cong \prod_{i=1}^d \OO_K/P_i^{e_i} = \prod_{i=1}^d (\OO_K)_{P_i}/P^{e_i}(\OO_K)_{P_i},\]
    where $(\OO_K)_{P_i}$ is the completion of $\OO_K$ at $P_i.$ By \cref{prop:product}, \cref{lem:product_coefficients} and \cref{lem:induced} and \cref{lem:local}, we conclude that $\OO_K/\mathcal{I}$ is a symmetric $\Z/n$-algebra.
\end{proof}
For function fields, we have an analogous statement. 

\begin{thm} \label{thm:function_fields}
Let $K$ be a finite extension of $\F_p(t)$ and $\OO_K$ its ring of integers. For each non-zero ideal $\mathcal{I} \subset \OO_K$, $\OO_K/\mathcal{I}$ is a symmetric $\F_p$-algebra. 
\end{thm}
By \cref{thm:main} and \cref{thm:function_fields}, we have the following corollary which answers a question posed in \cite[Remark 6.22]{minavc2024gcd}.  
\begin{cor}
Let $K$ be a finite extension of $\F_p(t)$ and $\OO_K$ its ring of integers. Let $\mathcal{I} \subset \OO_K$ be a non-zero ideal and $R=\OO_K/\mathcal{I}$. Let $S \subset R \setminus \{0 \}$ be a subset such that $S=-S$. Then $\Gamma(R,S)$ is an integral graph if and only if $S \cup \{0 \}$ is an $\F_p$-vector subspace of $R.$
\end{cor}

\subsection{Integral generalized Paley graphs}
Let $\OO_K$ be the ring of integers of a number field $K.$ Let $\mathcal{I}$ be a non-zero ideal of $\OO_K$. Suppose  $\chi\colon (\OO_{K}/\mathcal{I})^{\times} \to \C^{\times}$ is a character such that $\chi(-1)=1$. The following definition is motivated by the definition of the quadratic Paley graphs as defined and studied in \cite{minac2023paley} and \cite[Section 4.2]{chudnovsky2024prime}. 
\begin{definition}
The Paley graph $P_{\chi}$ associated with $\chi$ is defined to be the Cayley graph $\Gamma(\OO_K/\mathcal{I}, \ker(\chi))$ where 
\[ \ker(\chi) = \{a \in (\OO_K/I)^{\times}| \chi(a)=1. \} \]
\end{definition}

By  \cref{thm:number_field} and \cref{thm:main} , we have the following. 

\begin{prop}
    $P_{\chi}$ is an integral graph if and only if the induced Dirichlet character of $\chi$ on $(\Z/n)^{\times}$ is trivial. Here $n$ is the positive integer such that 
    \[ n \Z = \Z \cap \mathcal{I}. \]
\end{prop} 

\begin{rem} \label{rem:future_work}
    There are various constructions of $\chi$ with the property that the induced Dirichlet character on $(\Z/n)^{\times}$ is trivial. For example, let us consider $K=\Z[i]$. For each $n \in \Z$, let $\chi_n: (\Z[i]/n)^{\times} \to \C^{\times}$ be the quartic residue symbol associated with $n$ (see \cite[Chatper 9]{ireland2013classical}). Then, by \cite[Proposition 9.8.3]{ireland2013classical}, we know that $\chi_n(a)=1 $ for all $a \in (\Z/n)^{\times}.$ As a result, the Paley graph $P_{\chi}$ is integral. 
    It seems to be interesting to study some further arithmetic properties of the eigenvalues of $P_{\chi}$ in this case.  For example, one may ask whether it is possible to calculate the spectrum of $P_{\chi}$ explicitly (the case in which $\chi$ is a quadratic Dirichlet character has been solved in \cite{minac2023paley}.) Furthermore, since $P_{\chi}$ is an integral graph, it is a promising candidate for the existence of perfect state transfers (see \cite{bavsic2013characterization, godsil2012state}). We, therefore, wonder whether it is possible to classify $P_{\chi}$ which admit perfect state transfers. We hope to discuss these questions in future work. 
\end{rem}

\section*{Acknowledgements}
We thank Minh-Tam Trinh for introducing us to symmetric algebras and for some clarifications on Hecke characters. We also thank  J\'an Min\'a\v{c} for his constant encouragement and support.  Finally, we thank the anonymous referee for their help in improving the quality and clarity of this paper.

\end{document}